\documentclass[a4paper,12pt,oneside]{amsart}

\usepackage{graphicx}
\usepackage{amssymb}  
\usepackage{latexsym} 
\usepackage{comment}
\usepackage{dsfont,color}
\usepackage{colonequals}
\usepackage{mathrsfs}

\usepackage[pdf,usenames,dvipsnames]{pstricks}
\usepackage{epsfig}
\usepackage{pst-grad} 
\usepackage{pst-plot} 
 \usepackage{ulem}

\usepackage{pstricks-add}
\usepackage{pst-poly}
\usepackage{subfig}

\usepackage{tikz}
\usetikzlibrary{matrix,arrows,decorations.pathmorphing}
\usepackage{tikz-cd}
\usepackage{tkz-berge}

\tikzset{
    partial ellipse/.style args={#1:#2:#3}{
        insert path={+ (#1:#3) arc (#1:#2:#3)}
    }
}
\usetikzlibrary{matrix, calc}

\usepackage{float}
\usepackage{subfig}

\definecolor{darkgreen}{rgb}{0,0.5,0}
\usepackage[
        colorlinks, citecolor=darkgreen,
        pdfauthor={Ingrid Bauer, Fabrizio Catanese},
        pdftitle={Rigidity of the Kummer covering},
        linktocpage        
]{hyperref}

\usepackage[alphabetic,backrefs,lite]{amsrefs} 

\usepackage[all]{xy} 

\usepackage{enumerate} 

\newcommand\sD{{\mathcal D}}

\newcommand\sP{{\mathcal P}}

\newcommand\sI{{\mathcal I}}

\newcommand\sM{{\mathcal M}}

\newcommand\la{\lambda}

\newcommand\e{\epsilon}
\newcommand\s{\sigma}

\newcommand\ga{\gamma}

\DeclareMathOperator{\Hom}{Hom}

\DeclareMathOperator{\divi}{div}

\DeclareMathOperator{\Sym}{Sym}

\newcommand{\CC}{\ensuremath{\mathbb{C}}}

\newcommand{\ZZ}{\ensuremath{\mathbb{Z}}}

\newcommand{\hol}{\ensuremath{\mathcal{O}}}

\newcommand{\PP}{\ensuremath{\mathbb{P}}}

\newcommand{\ra}{\ensuremath{\rightarrow}}

\def\eea{\end{eqnarray*}}
\def\bea{\begin{eqnarray*}}

\newcommand\dual{\mathrel{\raise3pt\hbox{$\underline{\mathrm{\thinspace d
\thinspace}}$}}}
\newcommand\qe{\ifhmode\unskip\nobreak\fi\quad $\Box$}       

\def\BOX{\hfill\lower.5\baselineskip\hbox{$\Box$}}

\newtheorem{theorem}{Theorem}[section]
\newtheorem{lemma}[theorem]{Lemma}
\newtheorem{corollary}[theorem]{Corollary}
\newtheorem{proposition}[theorem]{Proposition}
\newtheorem{prop}[theorem]{Proposition}

\theoremstyle{remark}
\newtheorem{remark}[theorem]{Remark}

\theoremstyle{definition}
\newtheorem{definition}[theorem]{Definition}

\newcommand{\SSS}{\ensuremath{\mathcal{S}}}

\DeclareMathOperator{\Aut}{Aut}

\DeclareMathOperator{\sgn}{sgn}

\DeclareMathOperator{\trace}{trace}

\DeclareMathOperator{\Aff}{Aff}

\numberwithin{equation}{section}

\newcounter{nootje}
\renewcommand\check[1]
  {\marginpar{\tiny\begin{minipage}{20mm}\begin{flushleft}\thenootje : #1\end{flushleft}\end{minipage}}\addtocounter{nootje}{1}}
\begin{document}

\title[Del Pezzo of Degree 5]{$\mathfrak S_5$-equivariant  Syzygies for the Del Pezzo Surface of Degree 5}

\author{Ingrid Bauer}
\address{Mathematisches Institut,
         Universit\"at Bayreuth,
         95440 Bayreuth, Germany.}
\email{Ingrid.Bauer@uni-bayreuth.de}

\author{Fabrizio Catanese}
\address{Mathematisches Institut,
         Universit\"at Bayreuth,
         95440 Bayreuth, Germany.}
\email{Fabrizio.Catanese@uni-bayreuth.de}
\date{\today}
\thanks{
The present work took place in the  framework  of  
the ERC-2013-Advanced Grant - 340258- TADMICAMT.  Part of the work was done while the authors were guests at KIAS,
the second author as KIAS Research scholar.}

 \makeatletter
    \def\@pnumwidth{2em}
  \makeatother

\maketitle
\addtocontents{toc}{\protect\setcounter{tocdepth}{1}}

\tableofcontents

\section{Introduction}

Del Pezzo \cite{dp} (see also \cite{conforto}, page 313) showed that the smooth nondegenerate surfaces $Y^2_n$ of degree $n$ in $\PP^n$,
except for the anticanonical embedding of $\PP^1 \times \PP^1 \ra \PP^8$, are obtained as the blow-up of the projective plane in $ (9 - n)$ points 
 of which no three are collinear, no six lie on a conic (here $ n = 3,4,5,6,7,8,9$). 
 After his classification 
these  surfaces  bear his name.

Del Pezzo surfaces  are projectively unique as soon as $n \geq 5$,
and we are  interested (see \cite{debart}) about their  defining equations,  here particularly in the case $n=5$.

It is well  known, at least since the work of Buchsbaum and Eisenbud, \cite{be} that all surfaces $Y$ in $\PP^5$ which are arithmetically Cohen-Macaulay and subcanonical 
(i.e., the canonical sheaf $\omega_Y \cong \hol_Y (r)$ for some integer $r$) are defined  by the $m$-Pfaffians of an antisymmetric 
$(2m+1)\times (2m+1)$-matrix $A = - ^t A$
of homogeneous forms (more precisely, the Pfaffians generate the ideal of polynomials vanishing on $Y$). 
The simplest nondegenerate case is the case $m=2$ of a generic $5\times 5$ antisymmetric matrix $A$ of linear forms: here $r=-1$ and we get the Del Pezzo surface $Y^2_5$
of degree $5$
as Pfafffian locus of $A$, hence defined by the $5$ quadratic equations which are the five $4 \times 4$-Pfaffians of $A$. 

It suffices to choose a generic matrix, but an explicit matrix is better, as was for instance done in \cite{schicho}. At any rate we are in this way far away from a normal form for the equations of $Y^2_5$, which might be desirable for many purposes.
What do we mean by a normal form? To explain this, recall that 
 $Y^2_5$  is indeed isomorphic to  the moduli space of ordered quintuples of points in $\PP^1$, 
and  this isomorphism   follows after  showing that its  group of automorphisms  is  $\mathfrak S_5$ ($Y^2_5$ is indeed  in bijection with the set of projective
equivalence classes of quintuples where no point of $\PP^1$ occurs with
multiplicity $\geq 3$). Hence, we would like to have equations where this symmetry shows up, and which are `invariant' for the symmetry group.

This requirement can be  precisely specified  as follows: we have the anticanonical (6-dimensional) vector space $ V : = H^0 (\omega_Y^{-1}) = H^0(Y, \hol_Y(-K_Y))$
and the 5-dimensional vector space $W$ of quadratic forms vanishing on $Y$ ($W \subset \Sym^2(V)$). Both vector spaces are representations of
the symmetric group $\mathfrak S_5$ and the $5\times 5$ antisymmetric matrix $A$ of linear forms is seen, by the Buchsbaum-Eisenbud theory,
to be an invariant  tensor 
$$  A \in ((\Lambda^2 W) \otimes V)^{\mathfrak S_5}.$$

It was known through character theory (\cite{dolgachev1} with corrections done in \cite{dolgachev2},  \cite{takagi1}, \cite{takagi}) that this invariant tensor is unique 
up to constants.

The main result of the present paper is to explicitly and canonically determine it. Now, it looks like there is no unique such representation of the tensor $A$,
because  in the end we have  still to choose a basis for both vector spaces $V$ and $W$. However,  if we take the natural  irreducible 
4-dimensional representation $U_4$ of  $\mathfrak S_5$,
this is given by the invariant subspace $x_1 + x_2 + \dots +  x_5=0$ in the 5-dimensional vector space $U'_5$ with coordinates $(x_1, \dots, x_5)$
which are naturally permuted by $\mathfrak S_5$. So, there is a natural permutation representation $ U'_5$
yielding a natural basis for $  U'_5  \supset U_4$.

In our case, we first show that there is a natural basis $s_{ij}$ ($1 \leq i \neq j \leq 3$) of $V$, showing that $V$ is the regular representation of $\mathfrak S_3$ tensored with the sign character, and then we show that the space
$W$ is a 5-dimensional invariant subspace in a natural 6-dimensional subspace $W'$ of $\Sym^2(V)$, related to a natural permutation representation of $\mathfrak S_5$
\footnote{if $\e$ is the sign character, indeed $W' \otimes \e$ is the permutation representation corresponding to double combinatorial pentagons (the cosets in $\mathfrak S_5/ \Aff(1 , \ZZ /5)$).},
and with basis $Q_{ij}$ again related to the regular representation of $\mathfrak S_3$ (the space $W$ is then generated by the differences between two such quadratic forms 
$Q_{ij}$).

This leads to a  normal form: we produce in Theorem \ref{mainresult} (which we reproduce here immediately below)  in an elegant numerical way an explicit antisymmetric $6\times 6$ matrix $A'$ with entries in $V$ (i.e., with entries  linear forms) 
whose 15 $4 \times 4$-Pfaffians  are  exactly twice the differences between the quadratic forms $Q_{ij}$.

\begin{theorem}
Let $Y$ be the del Pezzo surface of degree $5$, embedded anticanonically in $\PP^5$. Then the ideal of $Y$ is generated by the $4 \times 4$ - Pfaffians of the $\mathfrak{S}_5$ - equivariant anti-symmetric $6 \times 6$-matrix $A'=$

$$
\begin{pmatrix}
0 & s_{21}+s_{23}- & s_{12}+s_{31}- & -s_{13}-s_{21} &s_{13}+s_{32}&s_{21}+s_{32}- \\
   & -s_{31}-s_{32} &-s_{32}-s_{21}&                          &                        & -s_{12}-s_{23} \\
   \hline
-s_{21}-s_{23}+ & 0 & s_{12}+s_{23}&s_{31}+s_{23}- &s_{13}+s_{21} -&-s_{12}-s_{31} \\
+s_{31}+s_{32} &  &                        &  -s_{13}-s_{32} & -s_{23}-s_{31} & \\
\hline
-s_{12}-s_{31}+ & -s_{12}-s_{23} & 0 & s_{12}+s_{13}- & s_{12}-s_{21}+& s_{23}+s_{31} \\
s_{32}+s_{21} &                          &   &-s_{32}-s_{31} & s_{31}-s_{13} & \\
\hline
 s_{13}+s_{21} & -s_{31}-s_{23}+ & -s_{12}-s_{13}+& 0 & -s_{21} - s_{32} & s_{23}+s_{12} - \\
                        & s_{13}+s_{32} & s_{32}+s_{31} &        &                           &-s_{13}-s_{32} \\
\hline
-s_{13}-s_{32}&-s_{13}-s_{21} +& -s_{12}+s_{21}-& s_{21} +s_{32} & 0 & s_{12}+s_{13}- \\
                       & s_{23}+s_{31} & -s_{31}+s_{13} &                           &  & -s_{21}-s_{23} \\
\hline
-s_{21}-s_{32}+ &s_{12}+s_{31} &-s_{23}-s_{31} & -s_{23}-s_{12} + &-s_{12}-s_{13}+ & 0 \\
s_{12}+s_{23} &                          &                        &s_{13}+s_{23} &s_{21}+s_{23} &
\end{pmatrix}
$$
\end{theorem}

The action of the symmetric group $\mathfrak{S}_5$ on the entries of the matrix is given by the action on the vector space
$V$, with basis $s_{ij}$, $i \neq j, 1 \leq i,j\leq3$, which is  described 
in the proof of theorem \ref{v}, tensored with the sign character $\e$ (we shall denote $V' : = V \otimes \e$ and observe
here  that $V$, $V'$ are isomorphic $\mathfrak{S}_5$ representations).  The action  on the matrix is determined
by the action of $\mathfrak{S}_5$ on $ W' \otimes W'$, in turn determined by the action of $\mathfrak{S}_5$
on the vector space $W'$ with basis the quadrics $Q_{ij}$, $i \neq j, 1 \leq i,j\leq3$. 
The latter action  (which is explicitly described in remark \ref{formulae}) is a consequence of the former, since
$Q_{ij} = s_{ij} \s_{ij}$, where the $\s_{ij}$'s are defined in section 1 and are as follows: $
\sigma_{ij} = \epsilon (s_{ik}-s_{ki}-s_{jk})$.

In practice, the action of two  generators of $\mathfrak{S}_5$ is explicitly given as follows.

The transposition $\tau:=(1,2)$ acts as
\begin{itemize}
\item $s_{ij} \mapsto s_{\tau(i)\tau(j)}$;
\item $Q_{ij} \mapsto -Q_{\tau(i)\tau(j)}$.
\end{itemize}

The cycle $\varphi:=(1,2,3,4,5)$ acts as

$$
s_{12}   \mapsto   s_{31}, \ s_{13}  \mapsto \sigma_{12}=s_{13}- s_{31}-s_{23},  \ s_{21}  \mapsto  s_{21},
$$
$$
s_{23}   \mapsto  -\sigma_{13}=s_{12}-s_{21}-s_{32}, \  s_{31} \mapsto  -\sigma_{32}=s_{13} -s_{31}+s_{21},  \ s_{32}  \mapsto  -\sigma_{23}=s_{12} -s_{21}+s_{31},
$$
and
$$
Q_{12}   \mapsto   Q_{31}, \ Q_{13}  \mapsto Q_{12},  \ Q_{21}  \mapsto  Q_{21},
$$
$$
Q_{23}   \mapsto  Q_{13}, \  Q_{31} \mapsto  Q_{32},  \ Q_{32}  \mapsto  Q_{23}.
$$

Our main theorem  is also interesting because, even if we work sometimes using character theory and over a field of characteristic $0$, in the end we produce equations which
define the Del Pezzo surface of degree $5$ over any field of characteristic different from 2 (in this case symmetric and antisymmetric tensors coincide,
and the situation should be treated separately).

We believe that our approach is also interesting from another point of view (as already announced in \cite{takagi1} and \cite{takagi}):
the symmetries of $Y^2_5$ bear some similarity to those of the icosahedron, as it is well known (\cite{dolgachev1}, \cite{cheltsov}), and the irreducible representations
of  $\mathfrak S_5$ can be explicitly described via the geometry of $Y^2_5$. 
For instance, the representations $V$, $W$ are irreducible, and together with the natural representation $U_4$ described above,
they and their tensor product with the sign representation $\e$  yield all irreducible representations of $\mathfrak S_5$  ($V$ is the only one such
that $ V \cong V \otimes \e$).

We also give  explicit descriptions of the irreducible representations of $\mathfrak S_5$ via the geometry of $Y^2_5$.
While it is easy to describe $U_4$ and $W \otimes \e $ through natural permutation representations on  
 $\mathfrak S_5 / \mathfrak S_4$,
respectively on $\mathfrak S_5 /\Aff(1, \ZZ/5)$, we give another realization of   the anticanonical vector space  $V$,
$$V =  H^0 (\omega_{Y^2_5}^{-1}) = H^0 (\hol_{  Y^2_5} (- K_{  Y^2_5}))  = H^0 (\hol_{  Y^2_5} (1)) , $$
through an explicit  permutation representation of dimension $24$, the one on the set of oriented combinatorial pentagons (the cosets in $\mathfrak S_5 /( \ZZ/5)$).
This is interesting, because we have then a bijection between the set of the 12 (unoriented, nondegenerate) combinatorial pentagons and the set of the 12 geometric 
nondegenerate pentagons
contained in $Y$, and formed with quintuples out of the 10 lines contained in $Y$ (these geometric pentagons are hyperplane sections of $Y$).

Finally, in a previous paper \cite{debart} we have written explicit equations of Del Pezzo surfaces in products $(\PP^1)^{h }$,
and here, for the Del Pezzo of degree $5$, we give in Theorem \ref{product} $\mathfrak S_5$-invariant equations inside $(\PP^1)^{5 }$.

An interesting phenomenon is that the minimal (multigraded) Hilbert resolution for this embedding of the Del Pezzo of degree $5$ is exactly the same as for the 
natural embedding in $(\PP^1)^{5 }$ of the Del Pezzo surfaces of degree $4$, so that we have an example of a  reducible Hilbert scheme.


\section{Symmetries and the anticanonical system of the Del Pezzo surface of Degree 5}\label{section1}

We  recall some notation introduced in \cite{takagi}, as well as   some  intermediate results established there.

 The Del Pezzo surface $Y: = Y^2_5$ of degree $5$ is the blow-up  of the plane $\PP^2$ in the $4$ points $p_1, \dots, p_4$
  of a projective basis, that is, we choose 
  $$
  p_1 =  (1:0:0), \ p_2=(0:1:0), \ p_3 = (0:0:1), \ p_4 = (1:1:1) = (-1:-1:-1).
  $$
Observe that $p_i$ corresponds to a vector $e_i$, for $i=1,2,3,4$ , where $e_1, e_2, e_3$ is a basis of a vector space $U'$, and $e_1  + e_2 + e_3 + e_4 = 0$.
In other words, $\PP^2 = \PP (U')$, where $U'$ is the natural irreducible representation of $\mathfrak S_4$.

As already mentioned in the introduction, the Del Pezzo surface $Y$  is indeed  the moduli space of ordered quintuples of points in $\PP^1$, 
and its  automorphism group  is isomorphic to $\mathfrak S_5$.

The obvious action of the symmetric group $\mathfrak S_4$ permuting the $4$ points extends in fact to an action
of the  symmetric group $\mathfrak S_5$. 

This can be seen as follows. The six lines in the plane joining pairs $p_i, p_j$ can be labelled as $L_{i,j} $, with $ i,j \in \{1, 2,3,4\}, i \neq  j $.

Denote by $E_{i,5}$ the exceptional curve lying over
the point $p_i$, and denote, for $i \neq   j \in \{1, 2,3,4\}$, by $E_{h,k} = E_{k,h}$ the strict transform in $Y$ of the  line $L_{i,j}$, if $\{1, 2,3,4\}=  \{i, j, h,k\}$.
For each choice of $3$ of the four points, $\{1, 2,3,4\} \setminus \{h\}$, consider the standard Cremona transformation $\s_h$
based on these three points. To it we associate the transposition $(h,5) \in \mathfrak S_5$, and the upshot is that
$\s_h$ transforms the $10$ $(-1)$ curves $E_{i,j}$ via the action of $(h,5) $ on pairs of elements in $\{1, 2,3,4,5\}$. 

There are  five geometric objects permuted by $\mathfrak S_5$: namely, $5$ fibrations $\varphi_i : Y \ra \PP^1$,
induced, for $1 \leq i \leq 4$, by the projection with centre $p_i$, and, for $i=5$, by the pencil of conics through
the $4$ points. Each fibration  is a conic bundle, with exactly three singular fibres, corresponding to the possible partitions of
type $(2,2)$ of the set $\{1, 2,3,4,5\} \setminus \{i\}$.

The intersection pattern of the curves $E_{i,j}$, which generate the Picard group of $Y$
is dictated by the simple rule (recall that $ E_{i,j}^2 = -1, \ \forall i \neq j$)
$$ E_{i,j} \cdot E_{h,k} = 1 \ \Leftrightarrow \{i,j\} \cap \{h,k\}= \emptyset, \ \ E_{i,j} \cdot E_{h,k} = 0   \Leftrightarrow \{i,j\} \cap \{h,k\}\neq  \emptyset, \{i,j\}.$$
In this picture the three singular fibres of $\varphi_1$ are $$E_{3,4} +  E_{2,5}, \ E_{2,4} +  E_{3,5}, \ E_{2,3} +  E_{4,5}.$$

The relations among the $E_{i,j}$'s in the Picard group come from the linear equivalences 
$E_{3,4} +  E_{2,5} \equiv E_{2,4} +  E_{3,5} \equiv  E_{2,3} +  E_{4,5}$ and their $\mathfrak S_5$-orbits.

An important observation is that  $Y$ contains exactly ten lines (i.e.,  ten
irreducible curves $E$ with $ E^2 = E K_Y = -1$), namely the lines $E_{i,j}$. 

Their intersection pattern is described by the Petersen graph, whose vertices correspond to the curves 
$E_{i,j}$, and whose edges correspond to intersection points (cf. figure \ref{fig1}).

\begin{center}
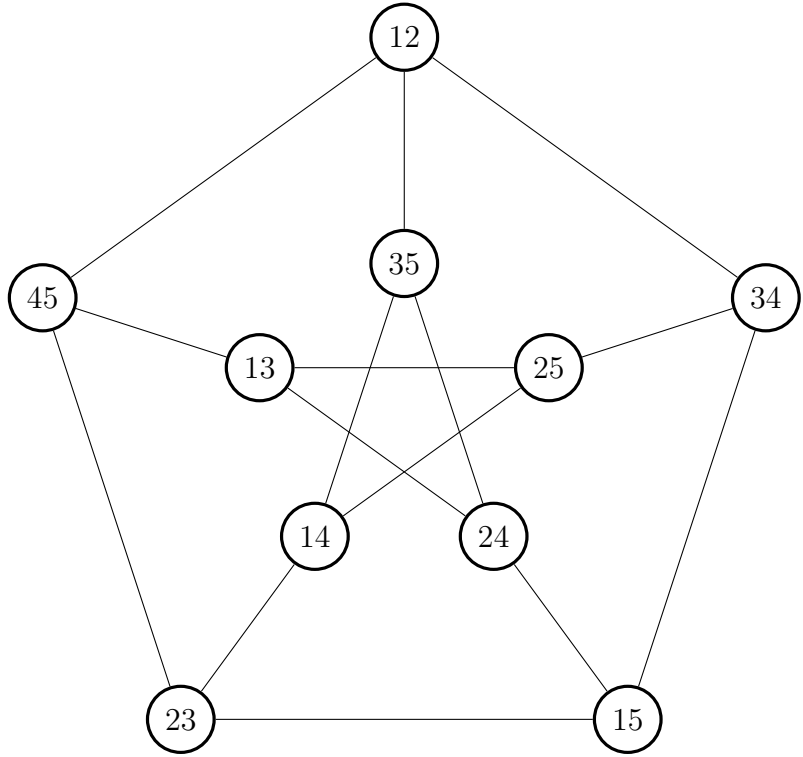
\begin{figure}
\begin{tikzpicture}[every node/.style={draw,circle,very thick}]
    \draw (18:5cm) node[label=18:] (2){34} 
     (90:5cm) node[label=90:] (3){12} 
     (162:5cm) node[label=162:] (4){45} 
      (234:5cm) node[label=234:] (5){23} 
     (306:5cm) node[label=306:] (1){15} ;
    
\draw (18:2cm) node[label=18:] (7){25} 
     (90:2cm) node[label=90:] (8){35} 
     (162:2cm) node[label=162:] (9){13} 
      (234:2cm) node[label=234:] (10){14} 
     (306:2cm) node[label=306:] (6){24}  (7);

\draw (1) -- (2);     
\draw (2) -- (3);
\draw (3) -- (4);
\draw (4) -- (5);
\draw (5) -- (1);

\draw (1) -- (6);
\draw (2) -- (7);
\draw (3) -- (8);
\draw (4) -- (9);
\draw (5) -- (10);

\draw (6) -- (8);
\draw (6) -- (9);
\draw (7) -- (9);
\draw (8) -- (10);
\draw (10) -- (7);
    
\end{tikzpicture}
\caption{The Petersen graph (the quotient of the dodecahedron via the antipodal map), which is the  dual 
graph to the incidence correspondence of the lines on the quintic Del Pezzo surface.}\label{fig1}

\end{figure}
\end{center}

\begin{remark} \
1) The action of $\Aut(Y) \cong \mathfrak{S}_5$ can be described as the action on $\PP^2$ by the following (birational) transformations:
\begin{itemize}
\item
the above vectorial action of $ \mathfrak{S}_4$ on $U'$
(by which  $\mathfrak{S}_4 \leq \mathfrak{S}_5$ acts on $\PP^2$ permuting $p_1,p_2,p_3,p_4$), i.e., in coordinates,
\item for $\sigma \in \mathfrak{S}_3$
$$
\sigma (x_1:x_2:x_3) = (x_{\sigma(1)}:x_{\sigma(2)}:x_{\sigma(3)}),
$$
and for $\sigma = (3,4)$:
$$
\sigma (x_1:x_2:x_3) = (x_1-x_3:x_2-x_3:-x_3),
$$

\item the transposition $(4,5)$ is the standard Cremona transformation 
$$(x_1:x_2:x_3) \mapsto (\frac{1}{x_1}:\frac{1}{x_2}:\frac{1}{x_3}) = (x_2x_3:x_1x_3:x_1x_2).
$$

\end{itemize}
2) The equation of the  six lines of the complete quadrangle in $\PP^2$ (the lines $L_{i,j}$ joining $p_i$ and $p_j$)
$$
\Sigma := x_1x_2x_3(x_1-x_2)(x_2-x_3)(x_3-x_1)
$$
yields an eigenvector for the sign representation $\e : \mathfrak{S}_5 \ra \{ \pm 1\}$.

In fact, the proper transform of $ \divi (\Sigma)$ is the sum of the 10 lines $E_{i,j}$ contained in $Y$, and   it is 
a divisor in $| - 2 K_Y|$  invariant for any automorphism. Since there is a natural action of $ \mathfrak{S}_5$
on $ \hol_Y(K_Y)$, we get a natural action of $ \mathfrak{S}_5$
on $H^0 ( Y, \hol_Y(-2 K_Y))$
hence  the proper transform of the equation $ \Sigma$ is a section of $H^0 ( Y, \hol_Y(-2 K_Y))$
which is an eigenvector for a character of $ \mathfrak{S}_5$.
There are only two characters, the sign and the trivial one: restricting to the subgroup $ \mathfrak{S}_3$
we see that it must be the sign.

\end{remark}

We consider now the anticanonical vector space

$$ V:= H^0(Y, \hol_Y(-K_Y)).$$

We know   that $V$ is, via adjunction,  naturally isomorphic to the vector space of cubics passing through the 4 points, 
$$ V':
= \{F \in \CC[x_1,x_2,x_3]_3 \ | \ F(p_1) =F(p_2) =F(p_3) =F(p_4) =0\} \cong \CC^6.$$

Since  $\mathfrak{S}_5$ acts linearly on $V$ we want to determine this action and translate this action on $V'$.

For this purpose we set,  for $1\leq i\neq j \leq 3$,
$$
s_{ij}:=x_i x_j(x_j-x_k), \ \{1,2,3\} = \{ i,j,k\}.
$$
\begin{prop}
$$
\{s_{ij} \ | \ 1\leq i\neq j \leq 3 \} \subset V'
$$
is  a basis of $V'$.
\end{prop}
\begin{proof}
All 6 cubic forms vanish on $p_1, \dots, p_4$. 

$s_{12}, s_{21}$ are the only ones not divisible by $x_3$, and they are independent modulo $(x_3)$,
since their divisors on the line $x_3 =0 $ are, respectively $ 2 p_1 + p_2$,  $  p_1 + 2 p_2$.

It suffices to show that the other 4 are independent after division by  $ x_3$. Now, $\frac{s_{23}} { x_3}, \frac{s_{32}}{ x_3}$
yield  $x_2 x_3$, respectively $x_2^2$ modulo $(x_1)$, hence it suffices to show that $\frac{s_{13} }{ x_3 x_1}, \frac{s_{31}}{ x_3 x_1}$
are independent: but these are respectively $(x_3 - x_2), (x_1 - x_2)$.

\end{proof}

We set 
$$
\sigma_{ij} := \frac{\Sigma}{s_{ij}} = \frac{\Sigma}{x_ix_j(x_j-x_k)} = \epsilon x_k(x_i-x_j)(x_k-x_i), 
$$
where $ \epsilon := \e (\begin{pmatrix} 1 & 2 & 3 \\i&j&k \end{pmatrix}$) is the sign of the permutation.

Then we have:
$$
\sigma_{ij} = \epsilon (s_{ik}-s_{ki}-s_{jk}).
$$
 
 \begin{remark}
 (1) We observe as a parenthetical remark  that the $\sigma_{ij}$'s  span a $4$-dimensional sub vector space of $V$. In fact, 
 $$
 \alpha (\sigma_{12} + \sigma_{13}) + \beta (\sigma_{21} + \sigma_{23}) +  \gamma (\sigma_{31} + \sigma_{32}) = 0,
 $$
 for all $\alpha, \beta, \gamma \in \CC$ such that $\alpha + \beta + \gamma = 0$.
 
 (2) Observe that for $\mathfrak{S}_3 \leq \mathfrak{S}_5$, $\mathfrak{S}_3$ acts on homogeneous polynomials by permuting $x_1,x_2,x_3$,
 and that $V'$ is the regular representation of $\mathfrak{S}_3$.
 
 \end{remark}
Consider now the Euler sequence on $\PP^2 := \PP(U')$, $U' \cong \CC^3$:
$$
0 \ra \hol_{\PP^2} \ra U' \otimes \hol_{\PP^2}(1) \ra \Theta_{\PP^2} \ra 0.
$$
From this follows  
$$
\hol_{\PP^2}(-K_{\PP^2}) \cong \wedge^3 U' \otimes \hol_{\PP^2}(3).
$$
Therefore the isomorphism $\CC[x_1,x_2,x_3]_3  \cong H^0(\PP^2, \hol_{\PP^2}(-K_{\PP^2}))$ is given by the map:
$$
P \mapsto P(dx_1\wedge dx_2 \wedge dx_3)^{-1}.
$$
 Hence we identify  $V'$ with $V$ via the map:
$V' \ra V, \ s_{ij} \mapsto s_{ij} (dx_1 \wedge dx_2 \wedge dx_3)^{-1}$. In this way we can read out on $V'$ the given action of $\mathfrak{S}_5$ on $V$.
We obtain immediately:
\begin{lemma}\label{s3} \
$\mathfrak{S}_3 \leq \mathfrak{S}_5$ acts on $V$ by
$$
\tau(s_{ij}) = \e(\tau) s_{\tau(i)\tau(j)}.
$$
Hence the representation $V$ restricted to $\mathfrak S_3 \leq \mathfrak S_5$ is the regular representation tensored by the sign character $\e$, while  $V$ restricted to  $\mathfrak{S}_4 \leq \mathfrak{S}_5$ is isomorphic to the subspace of $S^3 ((U')^{\vee}) \otimes \e$
corresponding to the polynomials vanishing in $p_1, \dots, p_4$.
\end{lemma}

\begin{proof}
Direct computation for $\mathfrak{S}_3$, for $\mathfrak{S}_4$ observe that the tranposition $(3,4)$ acts sending 
\begin{equation}\label{S4} 
dx_1 \wedge dx_2  \wedge dx_3 \mapsto d(x_1-x_3) \wedge d(x_2-x_3) \wedge d(-x_3)= - dx_1 \wedge dx_2  \wedge dx_3.
\end{equation}
\end{proof}
\begin{remark}\label{sigma}
Since $s_{ij}\sigma_{ij} = \Sigma$ and  $\Sigma$ is an eigenvector for the sign representation, the previous lemma implies immediately that for $\tau \in \mathfrak{S}_3 \leq \mathfrak{S}_5$ we have:
$$
\tau(\sigma_{ij}) = \sigma_{\tau(i)\tau(j)}.
$$
\end{remark}

We prove the following

\begin{theorem}\label{v} \

1) The character vector $\chi_V$  of the $\mathfrak{S}_5$-representation $V=H^0(Y, \hol_Y(-K_Y))$ is equal to $(6,0,-2,0,0,1,0)$. Therefore $H^0(Y, \hol_Y(-K_Y))$ is the unique six dimensional irreducible representation of $\mathfrak{S}_5$.

2) The character of $\bigwedge^6 V$ is the sign character $\e = \chi_2$.
\end{theorem}

For the convenience of the reader and for the purpose of fixing the notation we recall below the character table of $\mathfrak{S}_5$.
Here, cf. \cite{james-liebeck},  pages 199-202, the natural permutation representation on $\ZZ/5$ yields $\chi_1 + \chi_4$, while $\Lambda^2 (\chi_4) = \chi_7$,
and $S^2 (\chi_4) = \chi_1 + \chi_4 + \chi_5$. Finally, $\chi_3 = \chi_4 \otimes \chi_2$, $\chi_6 = \chi_5 \otimes \chi_2$, and obviously $\chi_7 = \chi_7 \otimes \chi_2$.

\begin{table}[h]
 \caption{Character table of $\mathfrak{S}_5$}
 \begin{tabular}{cccccccc}
Conj. Class & $1$ & $(1,2)$ & $(1,2)(3,4)$&$(1,2,3)$ &$(1,2,3,4)$&$(1,2,3,4,5)$&$(1,2,3)(4,5)$ \\
\hline
\hline
$\chi_1 = \chi_{W_1}$& 1&1&1&1&1&1&1\\
\hline
$\chi_2 = \chi_{W_2} = : \e$& 1&-1&1&1&-1&1&-1\\
\hline
$\chi_3= \chi_{W_3}$& 4&-2&0&1&0&-1&1\\
\hline
$\chi_4= \chi_{W_4}$& 4&2&0&1&0&-1&-1\\
\hline
$\chi_5= \chi_{W_5}$& 5&1&1&-1&-1&0&1\\
\hline
$\chi_6= \chi_{W_6}$& 5&-1&1&-1&1&0&-1\\
\hline
$\chi_7= \chi_{W_7}$& 6&0&-2&0&0&1&0\\
\hline
\hline
 \end{tabular}
 \label{tab:meinetabelle}
 \end{table}
 \begin{remark}
 Indeed, as we shall also see later, $V = W_7$, $W = W_5$ and $U_4 = W_4$.
 \end{remark}
\begin{proof} \

1) We have already observed  that for $\tau \in \mathfrak{S}_3$ we have $\tau(s_{ij}) = \e(\tau)s_{\tau(i)\tau(j)}$ .

Recall that the transposition $\tau:=(3,4)$ acts as 
$$
\tau \begin{pmatrix} x_1\\x_2\\x_3 \end{pmatrix} = \begin{pmatrix} x_1-x_3\\x_2-x_3\\-x_3 \end{pmatrix}. 
$$
We use formula \ref{S4}, obtaining that $(3,4)$ acts by:
\begin{itemize}
\item $s_{12} = x_1x_2(x_2-x_3) \mapsto - (x_1-x_3)(x_2-x_3)x_2 = \sigma_{31}$;
\item $s_{13} =x_1x_3(x_3-x_2) \mapsto - (x_1-x_3)(-x_3) (-x_2) = s_{23}$;
\item $s_{21} = x_2x_1(x_1-x_3) \mapsto - (x_2-x_3)(x_1-x_3)x_1 = -\sigma_{32}$;
\item $s_{23} \mapsto s_{13}$;
\item $s_{31} = x_1x_3(x_1-x_2) \mapsto - (x_1-x_3)(-x_3)(x_1-x_2) = -\sigma_{12}$;
\item $s_{32} = x_2x_3(x_2-x_1) \mapsto - (x_2-x_3)(-x_3)(x_2-x_1) = \sigma_{21}$.
\end{itemize}
Next we consider the action of $\tau = (4,5)$. This time we do not have an action on the space of polynomials, hence
we work  on the affine chart $\{x_3=1\}$.

 Then: $\tau\begin{pmatrix} x_1\\x_2 \end{pmatrix} = \begin{pmatrix} \frac{1}{x_1}\\ \frac{1}{x_2} \end{pmatrix}$ and we have that $(4,5) $ acts by:
$$
\tau((dx_1\wedge dx_2)^{-1}) = (d \frac{1}{x_1}\wedge d \frac{1}{x_2})^{-1} = (\frac{1}{x_1^2}dx_1 \wedge \frac{1}{x_2^2}dx_2)^{-1} = x_1^2x_2^2 (dx_1\wedge dx_2)^{-1}.
$$
Therefore (again here $\{i,j,k\} = \{1,2,3\}$)
\begin{multline}
s_{ij}(dx_1\wedge dx_2)^{-1} = x_ix_j(x_j-x_k) (dx_1\wedge dx_2)^{-1} \mapsto \frac{1}{x_i}  \frac{1}{x_j}  \frac{x_k - x_j}{x_jx_k}x_1^2x_2^2 (dx_1\wedge dx_2)^{-1} = \\
= \frac{x_1^2x_2^2x_3^2}{x_ix_j^2x_k}(x_k-x_j) (dx_1\wedge dx_2)^{-1}  = x_ix_k(x_k-x_j) (dx_1\wedge dx_2)^{-1} = s_{ik} (dx_1\wedge dx_2)^{-1} .
\end{multline}
Now we are ready to calculate the trace of $\tau$ for 
$$\tau \in \{(1,2),(1,2)(3,4),(1,2,3),(1,2,3,4), \\ (1,2,3,4,5),(1,2,3)(4,5) \}.
$$

\underline{$\tau = (1,2)$:}
$$
s_{12}   \leftrightarrow   -s_{21}, \ s_{13}  \leftrightarrow  -s_{23},  \ s_{31}  \leftrightarrow  -s_{32},
$$
hence $\chi_V(\tau) = \trace(\tau) = 0$.

\underline{$\tau = (1,2)(3,4)$:}
$$
s_{12}   \mapsto   \sigma_{32} = s_{21}+s_{13}-s_{31}, \ s_{13}  \mapsto -s_{13},  \ s_{21}  \mapsto  -\sigma_{31} = s_{12} +s_{23}-s_{32},
$$
$$
s_{23}   \mapsto  -s_{23}, \  s_{31} \mapsto  -\sigma_{21} = s_{23} -s_{13}-s_{32},  \ s_{32}  \mapsto  \sigma_{12} = s_{13} -s_{31}-s_{23},
$$
hence $\chi_V(\tau) = \trace(\tau) = -2$.

 If we use the character table of $\mathfrak S_5$ we can now immediately conclude, since $V$ has dimension $6$ and  the trace of 
$(1,2)(3,4)$ is nonnegative for all other representations, except that $\chi_7( (1,2)(3,4)) = -2$.

For completeness we calculate explicitly the action of the other representatives of all the conjugacy classes (thereby giving a selfcontained proof that $V$ is irreducible).

\underline{$\tau = (1,2,3)$:}
$$
s_{12}   \mapsto   s_{23}, \ s_{13}  \mapsto s_{21},  \ s_{21}  \mapsto  s_{32},
$$
$$
s_{23}   \mapsto  s_{31}, \  s_{31} \mapsto  s_{12} ,  \ s_{32}  \mapsto  s_{13},
$$
hence $\chi_V(\tau) = \trace(\tau) = 0$.

\underline{$\tau = (1,2,3,4)$:}
$$
s_{12}   \mapsto  \sigma_{12} =s_{13}-s_{31}-s_{23} , \ s_{13}  \mapsto s_{31},  \ s_{21}  \mapsto   -\sigma_{13} =s_{12} -s_{21}-s_{32},
$$
$$
s_{23}   \mapsto  s_{21}, \  s_{31} \mapsto  -\sigma_{23}= s_{12} -s_{21}+s_{31},  \ s_{32}  \mapsto  \sigma_{32}=s_{13} -s_{31}+s_{21},
$$
hence $\chi_V(\tau) = \trace(\tau) = 0$.

\underline{$\tau = (1,2,3,4,5)$:}
$$
s_{12}   \mapsto   s_{31}, \ s_{13}  \mapsto \sigma_{12}=s_{13}- s_{31}-s_{23},  \ s_{21}  \mapsto  s_{21},
$$
$$
s_{23}   \mapsto  -\sigma_{13}=s_{12}-s_{21}-s_{32}, \  s_{31} \mapsto  -\sigma_{32}=s_{13} -s_{31}+s_{21},  \ s_{32}  \mapsto  -\sigma_{23}=s_{12} -s_{21}+s_{31},
$$
hence $\chi_V(\tau) = \trace(\tau) = 1$.

\underline{$\tau = (1,2,3)(4,5)$:}

$$
s_{12}   \mapsto   s_{21}, \ s_{13}  \mapsto s_{23},  \ s_{21}  \mapsto  s_{31},
$$
$$
s_{23}   \mapsto  s_{32}, \  s_{31} \mapsto  s_{13} ,  \ s_{32}  \mapsto  s_{12},
$$
hence $\chi_V(\tau) = \trace(\tau) = 0$.

 The above explicit calculations  show that the character vector for the representation $V$ is 
$$
\chi_V = (6,0,-2,0,0,1,0).
$$
hence $\langle \chi_V , \chi_V\rangle = 1$ and $V$ is an irreducible representation.

2) A basis of $\bigwedge^6V$ consists of the element $\omega:=s_{12} \wedge s_{13} \wedge s_{21} \wedge s_{23} \wedge s_{31} \wedge s_{32}$. Then $\tau = (1,2)$ maps $\omega$ to
$$
s_{21} \wedge s_{23} \wedge s_{12} \wedge s_{13} \wedge s_{32} \wedge s_{31} = -\omega,
$$
whence the claim follows.
\end{proof}

\section{Quadratic equations for the Del Pezzo surface of degree 5 and geometrical descriptions of $\mathfrak S_5$-representations } 
We observe that by Riemann-Roch and Kodaira vanishing we have:
$$
 h^0(\hol_Y(-2K_Y)) = \chi(\hol_Y(-2K_Y)) = \frac12(-2K_Y)(-3K_Y) +1 = 16.
$$
 The above formula can also be checked by direct calculation of  the space of homogeneous polynomials $F(x_1, x_2, x_3)$ of 
degree 6 vanishing of multiplicity 2 at the points  $p_1, \dots, p_4$. This direct calculation also shows that 
 the natural   morphism of $\mathfrak{S}_5$-representations:
$$
\Phi \colon \Sym^2(V) \cong \CC^{21} \ra H^0(Y, \hol_Y(2)) (=  H^0(Y, \hol_Y(-2K_Y))) \cong \CC^{16}.
$$
is surjective.
 
We set $W:= \ker(\Phi) (\cong \CC^5) = H^0(Y, \mathcal{I}_Y(2))$.

Observe that, setting $Q_{ij}:= s_{ij} \sigma_{ij}$, $1 \leq i\neq j \leq 3$, we have that $\Phi(Q_{ij}) = \Sigma \in H^0(Y, \hol_Y(2))$. It is easy to verify that the set
$\{Q_{ij} \ | \ 1 \leq i\neq j \leq 3 \}$ is $\CC$-linearly independent and therefore 
$$
q_{12}:=Q_{12} -Q_{32}, \ q_{13}:=Q_{13} -Q_{32}, \ q_{21}:=Q_{21} -Q_{32}, 
$$
$$
q_{23}:=Q_{23} -Q_{32}, \ q_{31}:=Q_{31} -Q_{32} 
$$
are linearly independent elements of $W$, hence form a basis.

If we denote by $W'$ the  vector subspace of $\Sym^2(V)$ spanned by the $Q_{ij} ,  \ 1 \leq i\neq j \leq 3$, then we have the following inclusions,
where $\sI$ is the sheaf of ideals of the functions vanishing on $Y$:
$$
H^0(Y, \sI_Y(2)) =W \subset W' = \langle Q_{ij} \ | \  \ 1 \leq i\neq j \leq 3 \rangle \subset \Sym^2(V).
$$

\subsection{Geometrical and combinatorial pentagons}

We want to give now some geometrical background to the choice of the sections $s_{ij}, \s_{ij}$, and the quadratic forms $Q_{ij}$,
which makes our calculations less mysterious.

\begin{definition}
Let $\SSS$ be a set. Then 

\begin{enumerate}
\item
an ordered {\it combinatorial } n-gon on $\SSS$ is a map $p: \ZZ/n \ra \SSS$; 
\item
the i-th {\it side} $L_i$ is the restriction of $p$ to $\{i, i+1\}$;
\item
an ordered combinatorial n-gon on $\SSS$ is said to be {\it nondegenerate} if $p$ is injective;
\item
an {\it oriented combinatorial n-gon} on $\SSS$ is an equivalence class of  ordered combinatorial n-gons on $\SSS$
for the action of $\ZZ/n$ given by composition on the source $ p (i) \sim p(i+a)$;
\item
a(n  unoriented) {\it  combinatorial n-gon} on $\SSS$ is an equivalence class of  oriented combinatorial n-gons on $\SSS$,
where $p(i) \sim p(-i)$ (that is, an equivalence class of ordered n-gons for the action of the dihedral group $D_n$, $ p (i) \sim p(\pm i+a)$);
\item
the {\it neighbouring  n-gon} of $p(i)$ is the   unoriented n-gon $p(2i)$;  it is nondegenerate, for $n$ odd, iff $p$ is nondegenerate, moreover 
in this case  it has no sides in common with the  n-gon  $p(i)$ 
once $ n \geq 5$;
\item
a {\it double combinatorial pentagon} ($n=5$) is the (unordered) pair of two neighbouring nondegenerate (unoriented) combinatorial pentagons.
\item
If $\SSS$ is a linear space, then a {\it geometrical} n-gon is the union of $n$ distinct  lines $L_i$ associated to 
a combinatorial pentagon $p(i)$ in such a way  that $L_i$ is the line joining $p(i)$ with $p(i+1)$.
\end{enumerate}
\end{definition}

\begin{remark}
If the set $\SSS$ has $n$ elements, and we consider only nondegenerate n-gons, then the objects in  item (1) are the elements of  the
set $\mathfrak S_n $,  those in item (4) are  the
cosets $\mathfrak S_n / (\ZZ/n)$, for  item (5) we get   the
cosets $\mathfrak S_n / D_n$, and finally for  item (7) we have  the
cosets $\mathfrak S_5 / \Aff(1, \ZZ/5)$. 
\end{remark}

\begin{proposition}
Let $Y$ be a Del Pezzo surface of degree 5.
\begin{enumerate}
\item[a)] The geometrical pentagons contained in $Y$ are
exactly 12.
They are given by  the divisors of zeros of the 12 sections $s_{ij}, \s_{ij}$. 

\item[b)] There is a bijection
between the set of such geometrical pentagons and the set of combinatorial nondegenerate pentagons on $\{1,2,3,4,5\}$.
This bijection associates to the combinatorial pentagon $[i \mapsto p(i)]$ the geometrical pentagon 
$$ E_{p(1)p(2)} +   E_{p(2)p(3)} +  E_{p(3)p(4)} +  E_{p(4)p(5)} +  E_{p(5)p(1)} .$$ 

\item[c)] There is a bijection between the set of the 6 quadrics $\{Q_{ij}\}$ and the set of double combinatorial pentagons
on $\{1,2,3,4,5\}$.

\item[d)] Moreover, the subset of $V$
$$ \sM : = \{ \pm s_{ij}, \pm \s_{ij} \}$$
is an orbit for the action of $\mathfrak S_5$ on $V$, and the stabilizer of $s_{21}$ is the cyclic subgroup 
generated by $(1,2,3,4,5)$.
In particular $\mathfrak S_4$ acts simply transitively on $\sM$. 
\end{enumerate}
\end{proposition}

Figures 2 and 3 illustrate items (6), (7), (b)  and (c)  above.

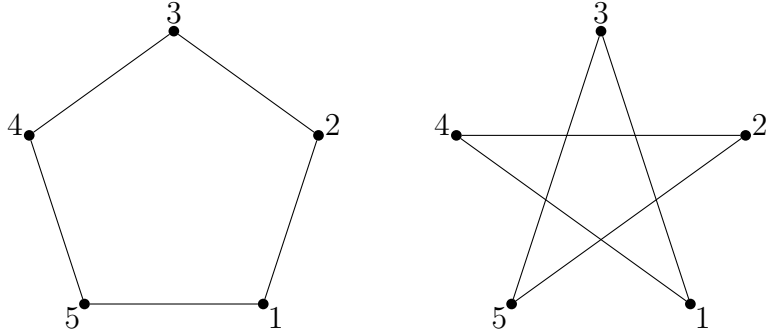
\begin{figure}
\centering
\subfloat{
\begin{tikzpicture}[bullet/.style={circle, fill,minimum size=4pt, inner sep=0pt, outer sep=0pt},nodes=bullet][baseline]
    \draw (18:2cm) node[label=18:2] (2){} 
     (90:2cm) node[label=90:3] (3){} 
     (162:2cm) node[label=162:4] (4){} 
      (234:2cm) node[label=234:5] (5){} 
     (306:2cm) node[label=306:1] (1){} ;

\draw (1) -- (2);     
\draw (2) -- (3);
\draw (3) -- (4);
\draw (4) -- (5);
\draw (5) -- (1);

\end{tikzpicture}
}
\qquad
\subfloat{
\begin{tikzpicture}[bullet/.style={circle, fill,minimum size=4pt, inner sep=0pt, outer sep=0pt},nodes=bullet][baseline]
    
    \draw (18:2cm) node[label=18:2] (2){} 
     (90:2cm) node[label=90:3] (3){} 
     (162:2cm) node[label=162:4] (4){} 
      (234:2cm) node[label=234:5] (5){} 
     (306:2cm) node[label=306:1] (1){} ;

\draw (1) -- (3);     
\draw (2) -- (4);
\draw (3) -- (5);
\draw (4) -- (1);
\draw (5) -- (2);

\end{tikzpicture}
}
\caption{A double combinatorial pentagon.}
\end{figure}

\begin{center}
\begin{figure}
\begin{tikzpicture}[every node/.style={draw,circle,very thick}]
    \draw (18:5cm) node[label=18:] (2){34} 
     (90:5cm) node[label=90:] (3){12} 
     (162:5cm) node[label=162:] (4){45} 
      (234:5cm) node[label=234:] (5){23} 
     (306:5cm) node[label=306:] (1){15} ;
    
\draw (18:2cm) node[label=18:] (7){25} 
     (90:2cm) node[label=90:] (8){35} 
     (162:2cm) node[label=162:] (9){13} 
      (234:2cm) node[label=234:] (10){14} 
     (306:2cm) node[label=306:] (6){24}  (7);

\draw (1) -- (2);     
\draw (2) -- (3);
\draw (3) -- (4);
\draw (4) -- (5);
\draw (5) -- (1);

\draw [dashed](1) -- (6);
\draw [dashed](2) -- (7);
\draw [dashed](3) -- (8);
\draw [dashed](4) -- (9);
\draw [dashed](5) -- (10);

\draw (6) -- (8);
\draw (6) -- (9);
\draw (7) -- (9);
\draw (8) -- (10);
\draw (10) -- (7);
    
\end{tikzpicture}
\caption{The pair of disjoint geometric pentagons in the Petersen graph associated to the above double combinatorial pentagon.}
\end{figure}
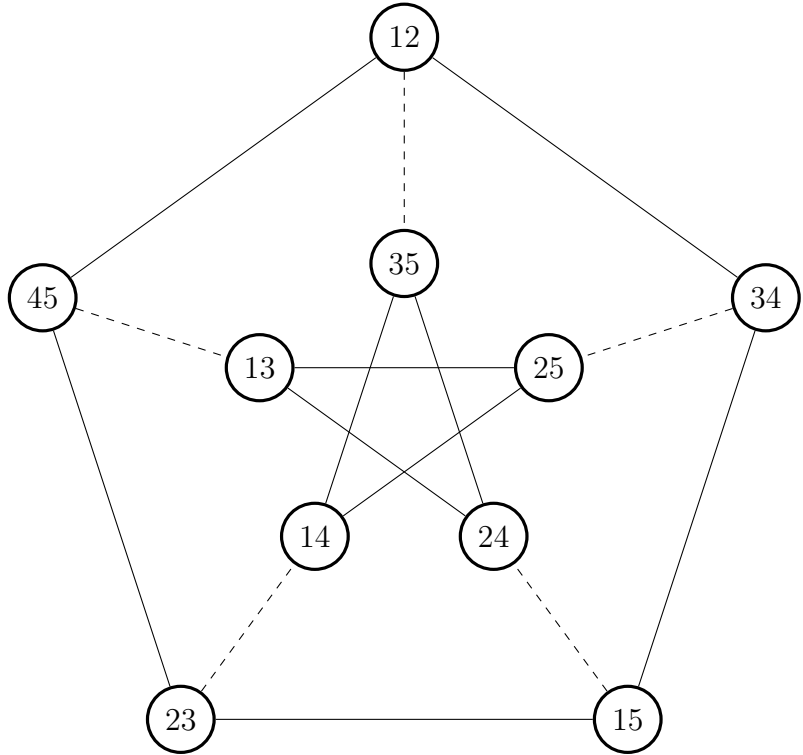
\end{center}

\begin{proof}
We know that the only lines contained in $Y$ are the 10 lines $E_{ij}$, and it is straightforward to verify that 
their union is the divisor $ \divi (\Sigma)$.

Let $L_1 + L_2 \dots + L_5$ be a geometrical pentagon; then $L_i$ intersects $L_{i+1}$, hence we get 
5 distinct pairs $A_i$ of elements in $\{1,2,3,4,5\}$ such that $A_i$ and $A_{i+1}$ are disjoint, $\forall i \in \ZZ/5$.

Hence $A_{i+1}$  is disjoint from the union of $A_i$ and $A_{i+2}$, in particular  $A_i$ and $A_{i+2}$ 
have exactly one element in common. 

We consider the  \footnote{  it is the dually neighbouring pentagon, where the dual n-gon is associated to the sequence of  sides $L_1, \dots, L_n$ } 5-gon $B_i$, where $B_i : = A_{2i}$.  We denote by 
$$p(i+1) : = B_i \cap B_{i+1} = A_{2i} \cap A_{2i+2}.$$ 
Hence $B_i = \{ p(i) ,p(i+1) \}$, and we get a nondegenerate combinatorial pentagon.

 Observe that we can recover the $A_i$'s from the $B_j$'s, simply letting $A_i : = B_{2i}$.  

By item (5) above the number of nondegenerate combinatorial pentagons on $(1,2,3,4,5)$ is $120 / 10 = 12$.

Direct inspection shows then  that $$ \divi (s_{12}) = E_{14} + E_{42} + E_{23} +E_{35} + E_{51} .$$
Similar formulae can be computed directly for $\divi (s_{ij})$, we may however observe that $\{ \divi (s_{ij}) \}$
is the $\mathfrak S_3$-orbit of  $\divi (s_{12})$, hence $\divi (s_{ij})$ is a geometric pentagon.

Since $\divi(\Sigma)$ is the union of the 10 lines of $Y$, follows that $\divi (\s_{ij})$ is associated to the neighbouring 
combinatorial pentagon of the one of $\divi (s_{ij})$.

Hence the quadrics $Q_{ij}$ correspond bijectively to double pentagons (these are 6).

The calculations performed above  in the proof of theorem \ref{v} show that if $\tau : = (1,2,3)(4,5)$, then $\langle \tau \rangle$ permutes 
 simply transitively the set $\{ s_{ij}\}$,   hence it sends the 
 set $\{ \s_{ij}\}$ to the set $\{ - \s_{ij}\}$.  While $\mathfrak S_3$ permutes the set $\sM' : = \{ \pm s_{ij}\}$, hence
also the set $\{ \pm \s_{ij}\}$.  Finally, $(3,4)$ permutes the set $\sM$. Since the above elements generate
$\mathfrak S_5$, $\sM$ is $\mathfrak S_5$-invariant.  The orbit of $s_{21}$ inside $\sM'$ for the subgroup generated
by   $\mathfrak S_3$ and $\tau : = (1,2,3)(4,5)$ has at least 7 elements, hence it equals $\sM'$. Since 
$(3,4)$ sends $s_{21}$ to $- \s_{32}$, follows that $\sM$ is a single  $\mathfrak S_5$-orbit.
Hence the stabilizer of $s_{21}$ has cardinality 5: but we know that it contains 
$(1,2,3,4,5)$.

 Indeed, $$ \divi (s_{21}) = E_{25} + E_{53} + E_{31} +E_{14} + E_{42} = \sum_{i=1}^5 E_{i, i+3}.$$

\end{proof}

We can summarize part of our discussion in the following corollary \footnote{a referee suggested a
  lengthy  alternative proof of this corollary using the isomorphism with the ring of invariants of five points on the line;
  here to an oriented  pentagon corresponds a  product of Pl\"ucker coordinates of a $2 \times 5 $ matrix,
  as done in the paper by Howard, Millson, Snowden, Vakil :  `The ideal of relations for the ring of invariants of $n$ points on the line', JEMS 14 (2012), 1, 1-60.}.

\begin{corollary}
There is a natural bijection between the set $\sM$ and the set of oriented nondegenerate combinatorial pentagons
on $\{1,2,3,4,5\}$, in such a way that  $-m$ is the oppositely oriented pentagon of $m$. $ \sM / \pm 1$ is then in bijection
with the set $\sP$ of pentagons,  where $\pm \s_{ij}$ corresponds to the neighbouring pentagon of $\pm s_{ij}$.

Defining
$$
\sigma(p)(i) := p(2i),
$$
we get  an order 4 transformation $\s$  on $\sM$, such that $\s^2 (m) = - m \ \forall m \in \sM$, and inducing an  involution  on $ \sM / \pm 1$, 
which exchanges $\pm s_{ij}$ with $ \pm \s_{ij}$,
hence such that $ ( \sM / \pm 1) / \s$ is in bijection with the set $\sD \sP$ of double pentagons.
\end{corollary}

\begin{proof}
 The bijection follows immediately from the fact that we have two transitive actions, and  the stabilizer of $s_{21}$ is the cyclic subgroup generated by $(1,2,3,4,5)$,
which is also the stabilizer of the standard pentagon corresponding to the identity map. 

We define then $\s (m)$ by the property that it associates to  an oriented pentagon $p(i)$ the neighbouring oriented pentagon $p(2i)$;
from this definition follows that $\s^2 (m) = - m$. Moreover, since $-m$ is the pentagon $p(-i)$, and $p(2(-i)) = p(- (2i))$
we obtain $ \s(- m)= - \s (m)$. Hence $\s(m) m =  \s(- m) (-m)$ inside $S^2(V)$. 
\footnote{Indeed $\s(s_{ij} )= \s_{ij}, \s(\s_{ij}) = - s_{ij}$. Observe moreover that, by  what we observed in remark 1.3,  $\s$ is not induced by a linear map of $V$!}

\end{proof}

Recall now that we have bijections:

\begin{itemize}
\item
$ \sM \cong \mathfrak S_5 / (\ZZ/5) =  \mathfrak S_5 / \langle (1,2,3,4,5)\rangle $ (oriented pentagons)
\item 
$ \sP \cong \sM  / \pm 1 \cong \mathfrak S_5 / D_5 =  \mathfrak S_5 / \langle (1,2,3,4,5), (1,4)(2,3) \rangle $ (pentagons)
\item
$ \sD \sP \cong ( \sM / \pm 1) / \s  \cong \mathfrak S_5 / \Aff(1, \ZZ/5) =  \mathfrak S_5 / \langle (1,2,3,4,5), (1,2,4,3) \rangle $ (double pentagons)
\end{itemize}

\begin{remark}
Observe that $\sigma$ corresponds to multiplication on the right by $(1,2,4,3)$ on the cosets $gH$, where $H$ is any of the above three subgroups.
\end{remark}
 Next we show that all irreducible representations of  $\mathfrak S_5$ of dimension different from 4 \footnote {These are easily gotten by the standard permutation
representation on 5 elements!} are contained in the permutation representation associated to the set $\sM$, i.e., to the set of oriented pentagons.

\begin{theorem}
 Let $M$ be the permutation representation associated to $\sM$, $D$ the permutation representation associated to   $ \sD \sP$, 
$P$ the permutation representation associated to   $  \sP$.

Then 
\begin{itemize}
\item $ W'  = W \oplus \epsilon = D \otimes \e$, and
\item
$ M = V \oplus (V \otimes \e) \oplus D \oplus (D \otimes \e) = V  \oplus V \oplus    [(W \otimes \e) \oplus  \chi_1] \oplus [W \oplus \e]$.
\end{itemize} 
\end{theorem}
\begin{proof}
First of all, $\s$ induces an action of $\ZZ/4$ on $M $, hence $M $ splits according to the 4 eigenvalues $ \pm \sqrt{-1}, \pm 1$.

The eigenspace for $+1$ is generated by the vectors (of $M$!) 
$$m + \s(m) + (-m ) + (- \s(m)),$$
 hence it clearly corresponds to the representation
$D$ on the double pentagons. 

The eigenspace $D'$ for $-1$ is generated by the vectors  
$$m - (\s(m)) + (-m ) - (- \s(m)),$$
 and together with the previous yields as direct sum
the $+1$-eigenspace for $-1 = \s^2$, which is clearly $P$, the representation associated to the pentagons. Observe that  $D'$  contains the sign representation $\e$.

The eigenspace $V'$  for $ \sqrt{-1}$ is generated by the vectors  $m -  \sqrt{-1}( \s(m)) -  (-m ) +  \sqrt{-1} (- \s(m))$. The Galois group of
$\la^2 = -1$ (complex conjugation) yields an isomorphism of $V'$ with the eigenspace $V''$ for $- \sqrt{-1}$.

{\em Step 1:} $V' \cong V$.

This follows immediately from Schur's Lemma: indeed $M$ surjects onto $V$, hence $V$ appears
as a summand of $M$. But $D$ and $D'$ are reducible, hence $P$ maps trivially to $V$,
and $V'$ is isomorphic to $V$. Therefore we have 
$$
M= D \oplus D' \oplus V^{\oplus 2},
$$
and we know that $V \cong V \otimes \e$.

{\em Step 2:} We calculate the character of $M$.

In order to achieve this, we use the following obvious 
\begin{lemma}
Consider a transitive  permutation representation of a finite group $G$ on $G/H$. Then its character $\chi_H$ is given by:
$$ \chi_H (\ga) = | \{ gH | \ga g H = g H \} |  = \frac{1}{|H|} | \{ g | \ga \in g H g^{-1} \} | .$$ 
\end{lemma}

In the case of the oriented pentagons, $ H $ is spanned by $(1,2,3,4,5)$ and 
$ \chi_H (\ga) = 0$ unless $\ga$ has order dividing $5$, in which case $  \chi_H  (Id) = 24$, while $  \chi_H ((1,2,3,4,5)) = 4$,
since the normalizer of $H$ is the affine group $ \Aff(1, \ZZ/5)$. 

Thus in this case the character vector equals 
$$\chi_{M}=  (24,0,0,0,0,4,0). $$

{\em Step 3:} $\chi_D = \chi_6 + \chi_1$, $D' = D \otimes \e$.

Subtracting twice the character of $V$, we obtain $\chi_{P}=  (12,0,4,0,0,2,0)$. Alternatively, and also in order to calculate the character of $D$, we again apply the above lemma.

If we consider  pentagons and double pentagons (i.e. $H= D_5$ resp. $H = \Aff(1, \ZZ /5)$),  observe that in both cases $g H g^{-1}$ contains no elements of order $3$ and in the case of pentagons no transpositions
and no 4-cycles.
Hence the character is zero,  in the case of pentagons, unless we have 5-cycles, double transpositions, or the identity.

 A 5-cycle fixes $2$ pentagons, and exactly one double pentagon. A double transposition fixes 4 pentagons, and 2 double pentagons.
 A 4-cycle fixes 2 double pentagons: for instance, if $ i \mapsto  2 i$   fixes $\{(0,a_1,a_2,a_3,a_4), (0,a_2,a_4,a_1,a_3)\}$,
 then it must preserve the block decomposition $\{1,2,3,4\} = \{a_1, a_4\} \cup \{a_2, a_3\}$ corresponding to the two sets of neighbours of $0$; since it does not fix any pentagon,
 it must exchange these two sets,  and then it must be either $a_2= 2 a_1$
 or $a_3 = 2 a_1$. Up to choosing an appropriate  representative, we may assume $a_1 = 1$; then it must be $a_4 = -1$, 
 and  the only two choices are $a_2 = 2$ or $a_2 = -2$.
 
We conclude from these  calculations that 
$$\chi_{P}=  (12,0,4,0,0,2,0),$$
$$\chi_{D}=  (6,0,2,0,2,1,0) = \chi_1+\chi_6$$ 

We recover in this way that 
$$\chi_{M} - \chi_{P} =  (12,0,-4,0,0,2,0) = 2 \chi_V =2  \chi_7. $$
and moreover we obtain:
$$\chi_{P} - \chi_{D} =   (6,0,2,0,-2, 1,0) = \chi_{D \otimes \e}.$$

This implies that $D' = D \otimes \e$.

{\em Step 4:} We shall show that $D \cong W' \otimes \e$ and $W' \cong W\oplus \e$.

In fact, $W' = \langle Q_{ij} \rangle$ surjects onto the one dimensional representation $\e$ corresponding to $\CC \Sigma \subset H^0(\hol_Y(2))$. In particular, $W' = W \oplus \epsilon$.

Observe now that, for each $\tau \in \mathfrak S_3$, 
$$
\tau(Q_{ij} ) = \epsilon(\tau) Q_{\tau(i)\tau(j)},
$$
while, defining 
$$
d_{ij}:= \left( (s_{ij}) + (-s_{ij}) + (\s_{ij}) + (-\s_{ij}) \right),
$$
 we have $\tau(d_{ij}) = d_{\tau(i)\tau(j)}$. 
 
 Hence $W'$ and $D \otimes \e$ have the same character on $\mathfrak S_3$. It follows that also $W'$ is the direct sum of an irreducible five dimensional representation with a one dimensional one, and either $W \cong D$ or $W \cong D \otimes \e$. The first possibility is excluded since $W'$ contains $\epsilon$.
 
\end{proof}

\begin{remark}\label{formulae}
We have indeed precise formulae for the action on the space $W'$.

By Lemma \ref{s3} and Remark \ref{sigma} we know that for $\tau \in \mathfrak{S}_3$ we have that
$\tau(Q_{ij}) = \sgn(\tau)Q_{\tau(i) \tau(j)}$.

In particular:

\underline{$\tau = (1,2)$:}
$$
Q_{12}   \leftrightarrow   -Q_{21}, \ Q_{13}  \leftrightarrow  -Q_{23},  \ Q_{31}  \leftrightarrow  -Q_{32}.
$$

From the proof of Theorem \ref{v} we see that  other elements of  $\mathfrak{S}_5$ act as follows:

\underline{$\tau = (3,4)$:}
$$
Q_{12}   \leftrightarrow   -Q_{31}, \ Q_{13}  \leftrightarrow  -Q_{23},  \ Q_{21}  \leftrightarrow  -Q_{32},
$$

\underline{$\tau = (4,5)$:}
$$
Q_{12}   \leftrightarrow   -Q_{13}, \ Q_{21}  \leftrightarrow  -Q_{23},  \ Q_{31}  \leftrightarrow  -Q_{32},
$$

\underline{$\tau = (1,2,3,4,5)$:}
$$
Q_{12}   \mapsto   Q_{31}, \ Q_{13}  \mapsto Q_{12},  \ Q_{21}  \mapsto  Q_{21},
$$
$$
Q_{23}   \mapsto  Q_{13}, \  Q_{31} \mapsto  Q_{32},  \ Q_{32}  \mapsto  Q_{23}.
$$

\end{remark}

For future use we also show:

\begin{lemma}
The determinant of $W$ is the trivial representation, equivalently, $\bigwedge^6 W' = \e$.
\end{lemma}

\begin{proof}
Since $W' \cong W \oplus \e$, we have that $\bigwedge^6 W' = \e$ if and only if $\bigwedge^5 W = \chi_1$.

A basis of $\bigwedge^6 W'$ consists of the element $\omega:=Q_{12} \wedge Q_{13} \wedge Q_{21} \wedge Q_{23} \wedge Q_{31} \wedge Q_{32}$. Then $\tau = (1,2)$ maps $\omega$ to
$$
Q_{21} \wedge Q_{23} \wedge Q_{12} \wedge Q_{13} \wedge Q_{32} \wedge Q_{31} = -\omega,
$$
whence the claim follows.

\end{proof}

\section{$\mathfrak{S}_5$-equivariant resolution of the sheaf of regular functions on the Del Pezzo surface $Y \subset \PP^5$ of degree $5$}

By Theorem \ref{v}, 2) we know that $\chi_{\bigwedge^6V} = \chi_2 = \e$.

Consider the Euler sequence on $\PP^5 := \PP(V)$, $V = H^0(Y, \hol_Y(-K_Y))$:
$$
0 \ra \Omega^1_{\PP^5} \ra V \otimes \hol_{\PP^5}(-1) \ra \hol_{\PP^5} \ra 0.
$$
Then 
$$
\bigwedge^6(V \otimes \hol_{\PP^5}(-1)) \cong \bigwedge^5  \Omega^1_{\PP^5} \otimes  \hol_{\PP^5},
$$

and this implies that
$$
\omega_{\PP^5} \cong (\det V) \otimes \hol_{\PP^5}(-6) \cong \hol_{\PP^5}(-6) \otimes \epsilon,
$$
 where $\epsilon$ is the sign representation.

We take a Hilbert resolution of $\hol_Y$ ($P:=\PP^5$):
\begin{equation}\label{hilbert}
0 \ra U \otimes \hol_P(-5) \ra W'' \otimes \hol_P(-3) \ra W \otimes \hol_P(-2) \ra  \hol_P \ra \hol_Y \ra 0,
\end{equation}
where $W,W'',U$ are $\mathfrak{S}_5$-representations  (a posteriori they shall be irreducible of  respective dimensions $5,5,1$,
at this stage we only know that $\dim(W) = 5$). 

Applying $\mathcal{H}om( \cdot , \omega_P)$ to the above resolution, we obtain a resolution of 
$$
\mathcal{E}xt^3(\hol_Y, \omega_P) \cong \omega_Y \cong \hol_Y(-1):
$$

\begin{equation}
0 \ra \omega_P \ra W^{\vee} \otimes \omega_P(2) \ra (W'')^{\vee} \otimes \omega_P(3) \ra  U^{\vee} \otimes \omega_P(5) \ra \hol_Y(-1) \ra 0.
\end{equation}

This equals
\begin{equation}
0 \ra \epsilon \otimes \hol_P(-6) \ra \epsilon \otimes W^{\vee} \otimes \hol_P(-4) \ra (W'')^{\vee} \otimes \epsilon \otimes \hol_P(-3) \ra  \epsilon \otimes U^{\vee} \otimes \hol_P(-1)
\ra \hol_Y(-1) \ra 0, 
\end{equation}

which after twisting by $\hol_P(1)$ becomes:
\begin{equation}\label{hilbdual}
0 \ra \epsilon \otimes \hol_P(-5) \ra \epsilon \otimes W^{\vee} \otimes \hol_P(-3) \ra (W'')^{\vee} \otimes \epsilon \otimes \hol_P(-2) \ra  \epsilon \otimes U^{\vee} \otimes \hol_P
\ra \hol_Y \ra 0.
\end{equation}

By the uniqueness of a minimal graded free resolution up to isomorphism, we get that (\ref{hilbert}) is isomorphic to (\ref{hilbdual}), hence:
$$
U \cong \epsilon, \ \ W'' \cong W^{\vee} \otimes \epsilon.
$$

Therefore we have proven the following:
\begin{prop}
The self dual  Hilbert resolution of $\hol_Y$ is given by:
\begin{equation}
0 \ra \epsilon \otimes \hol_P(-5) \stackrel{B^{\vee}}\ra \epsilon \otimes W^{\vee} \otimes \hol_P(-3) \stackrel{A} \ra W  \otimes \hol_P(-2) \stackrel{B} \ra   \hol_P \ra \hol_Y \ra 0,
\end{equation}
where $A \in W \otimes (W \otimes \epsilon) \otimes V$, $B \in W^{\vee} \otimes S^2(V) = \Hom(W,S^2(V))$, and $A$ is anti symmetric.

\end{prop}

Observe that $W \otimes (W \otimes \epsilon) \otimes V \cong W \otimes W \otimes  V$ as representations, since $V \cong V \otimes \epsilon$. Therefore we look for $A \in (\wedge^2 W \otimes V )^{\mathfrak{S}_5}$, and we claim that $A$ is uniquely determined up to scalars.

\begin{lemma}
The natural inclusion 
$$
\wedge^2 W \otimes V \subset \wedge^2 W' \otimes V
$$
induces an isomorphism
$$(\wedge^2 W \otimes V )^{\mathfrak{S}_5} \cong (\wedge^2 W' \otimes V )^{\mathfrak{S}_5} \cong \CC .
$$
\end{lemma}

\begin{proof}
Observe that $\wedge^2 W' \cong \wedge^2(W \oplus W_2) \cong \wedge^2 W \oplus (W \otimes W_2)\cong \wedge^2 W \oplus W_6  $, where $W_6 \cong W \otimes \e$ and $\wedge^2W \cong V \oplus W_3$ (as it is easily checked by the character formula
$$
\chi_{\wedge^2 W}(g) = \frac 12 (\chi_W(g)^2-\chi_W(g^2)). )
$$

By Schur's lemma it follows that 
$$
(\wedge^2 W' \otimes V )^{\mathfrak{S}_5} \cong ((\wedge^2 W \oplus W_6) \otimes V)^{\mathfrak{S}_5} \cong (\wedge^2 W \otimes V)^{\mathfrak{S}_5} \cong
$$
$$
\cong ((V \oplus W_3)\otimes V)^{\mathfrak{S}_5} \cong (V\otimes V)^{\mathfrak{S}_5} \cong \CC .
$$
\end{proof}

We want to find the (up to a constant) unique $\mathfrak{S}_5$-equivariant map $A$, such that if $A$ corresponds to a tensor $\alpha \in (\bigwedge^2 W \otimes V)^{\mathfrak{S}_5}$, then $\alpha \wedge \alpha = B$.

We have that $\alpha \wedge \alpha \in \bigwedge^4 W \otimes S^2 V$ and, since $\bigwedge^5 W \cong \CC$ is the trivial representation, hence the pairing
$$
\bigwedge^4 W \times W \rightarrow \bigwedge^5 W \cong \CC
$$ identifies $\bigwedge^4 W$
 with $W^{\vee}$, therefore
$$
\alpha \wedge \alpha = B \in W^{\vee} \otimes S^2(V) .
$$

Since we want to find the (up to a constant) unique element $A \in (\bigwedge^2 W' \otimes V )^{\mathfrak{S}_5}$,  
we write for $1 \leq i \neq j \leq 3$, $1 \leq h \neq k \leq 3$, $1 \leq m \neq n \leq 3$,:
$$
A := \sum_{ij,hk,mn} a_{ij,hk,mn} (Q_{ij} \wedge Q_{hk}) \otimes s_{mn}, \ A_{ij,hk} := \sum_{mn} a_{ij,hk,mn} s_{mn},
$$
and use the lexicographical order for $(ij)$, resp. $(hk)$, resp $(mn)$.

Moreover, we use that $A_{ij,hk}$ is skew-symmetric, i.e. $A_{ij,hk} = - A_{hk,ij}$.

We are going to prove the following:

\begin{theorem} \label{mainresult}
Let $Y$ be the del Pezzo surface of degree $5$, embedded anticanonically in $\PP^5$. Then the ideal of $Y$ is generated by the $4 \times 4$ - Pfaffians of the $\mathfrak{S}_5$ - equivariant anti-symmetric $6 \times 6$-matrix $A=$

$$
\begin{pmatrix}
0 & s_{21}+s_{23}- & s_{12}+s_{31}- & -s_{13}-s_{21} &s_{13}+s_{32}&s_{21}+s_{32}- \\
   & -s_{31}-s_{32} &-s_{32}-s_{21}&                          &                        & -s_{12}-s_{23} \\
   \hline
-s_{21}-s_{23}+ & 0 & s_{12}+s_{23}&s_{31}+s_{23}- &s_{13}+s_{21} -&-s_{12}-s_{31} \\
+s_{31}+s_{32} &  &                        &  -s_{13}-s_{32} & -s_{23}-s_{31} & \\
\hline
-s_{12}-s_{31}+ & -s_{12}-s_{23} & 0 & s_{12}+s_{13}- & s_{12}-s_{21}+& s_{23}+s_{31} \\
s_{32}+s_{21} &                          &   &-s_{32}-s_{31} & s_{31}-s_{13} & \\
\hline
 s_{13}+s_{21} & -s_{31}-s_{23}+ & -s_{12}-s_{13}+& 0 & -s_{21} - s_{32} & s_{23}+s_{12} - \\
                        & s_{13}+s_{32} & s_{32}+s_{31} &        &                           &-s_{13}-s_{32} \\
\hline
-s_{13}-s_{32}&-s_{13}-s_{21} +& -s_{12}+s_{21}-& s_{21} +s_{32} & 0 & s_{12}+s_{13}- \\
                       & s_{23}+s_{31} & -s_{31}+s_{13} &                           &  & -s_{21}-s_{23} \\
\hline
-s_{21}-s_{32}+ &s_{12}+s_{31} &-s_{23}-s_{31} & -s_{23}-s_{12} + &-s_{12}-s_{13}+ & 0 \\
s_{12}+s_{23} &                          &                        &s_{13}+s_{23} &s_{21}+s_{23} &
\end{pmatrix}
$$
\end{theorem}

We first prove the following:

\begin{proposition}\label{invariantmatrix} \
For $\{i,j,k\} = \{1,2,3\}$ we have:
\begin{enumerate}
\item $A_{ij,ik} = s_{ji} +s_{jk}-s_{ki}-s_{kj}$;
\item $A_{ik,kj} = -s_{ij}-s_{ki}$;
\item $A_{ij,ki} = s_{ik}+s_{kj}$;
\item $A_{ik,ki} = s_{ik} +s_{ji}-s_{jk}-s_{ki}$;
\item $A_{ij,kj} = s_{ji}-s_{ij}+s_{kj}-s_{jk}$.
\end{enumerate}
\end{proposition}

\begin{proof}
Observe that $\tau = (4,5)$ sends $Q_{12} \wedge Q_{13}$ to $-Q_{12} \wedge Q_{13}$. Therefore $A_{12,13} = -\tau A_{12,13}$, since $A$ is $\mathfrak{S}_5$ - invariant. Write
$$
A_{12,13} = \alpha_{12}s_{12} + \alpha_{13}s_{13} + \alpha_{21}s_{21} + \alpha_{23}s_{23} + \alpha_{31}s_{31} + \alpha_{32}s_{32},
$$
then we have 
\begin{multline}
A_{12,13} =  \alpha_{12}s_{12} + \alpha_{13}s_{13} + \alpha_{21}s_{21} + \alpha_{23}s_{23} + \alpha_{31}s_{31} + \alpha_{32}s_{32} = \\
-\tau A_{12,13} = \alpha_{12}s_{13}+ \alpha_{13}s_{12} + \alpha_{21}s_{23} + \alpha_{23}s_{21} + \alpha_{31}s_{32} + \alpha_{32}s_{31}.
\end{multline}
This implies
$$
\alpha_{12} = \alpha_{13}, \  \alpha_{21} = \alpha_{23}, \  \alpha_{31} = \alpha_{32}.
$$
On the other hand, if we apply $\tau = (2,3)$, we get:
\begin{multline}
A_{12,13} =  \alpha_{12}s_{12} + \alpha_{13}s_{13} + \alpha_{21}s_{21} + \alpha_{23}s_{23} + \alpha_{31}s_{31} + \alpha_{32}s_{32} = \\
-\tau A_{12,13} = -\alpha_{12}s_{13}- \alpha_{13}s_{12} - \alpha_{21}s_{31} - \alpha_{23}s_{32} - \alpha_{31}s_{21} - \alpha_{32}s_{23},
\end{multline}
hence 
$$
\alpha_{12} = \alpha_{13} = 0, \  \alpha_{21} = \alpha_{23} = -  \alpha_{31} =  - \alpha_{32}.
$$
W.l.o.g. we set $\alpha_{21} = 1$ and we have
\begin{equation}\label{a1213}
A_{12,13} = s_{21} +s_{23}-s_{31}-s_{32}.
\end{equation}

Assume that $\{i,j,k\} = \{1,2,3\}$ and let $\sigma:=\begin{pmatrix} 1&2&3 \\ i&j&k \end{pmatrix}$. Apply $\sigma$ to equation (\ref{a1213}), then we see that
$$
A_{ij,ik} = s_{ji} +s_{jk}-s_{ki}-s_{kj}.
$$
In particular we have: $A_{21,23} = s_{12} +s_{13}-s_{31}-s_{32}$. If we apply $\tau = (3,4)$, we obtain:
\begin{multline}
A_{13,32} =  -\tau A_{21,23} = -\tau (s_{12} +s_{13}-s_{31}-s_{32}) = \sigma_{31} +s_{23}-\sigma_{21} + \sigma_{12} =\\
s_{32}-s_{23}-s_{12} +s_{23}-s_{32}+s_{23}-s_{13}+s_{13}-s_{31}-s_{23} = -s_{12}-s_{31}.
\end{multline}
Applying $\sigma$ we get:
$$
A_{ik,kj} = -s_{ij}-s_{ki}.
$$
For $\tau = (4,5)$ we have $A_{12,31} = \tau A_{13,32} = s_{13}+s_{32}$, and applying $\sigma$ we get:
$$
A_{ij,ki} = s_{ik} + s_{kj}.
$$
We apply $\tau =(3,4)$ to the equation $A_{12,23} = -s_{13}-s_{21}$ and obtain
$$
A_{13,31} = - \tau A_{12,23} = \tau(s_{13}+s_{21}) = -s_{23}+ \sigma_{32} = -s_{23} +s_{13}-s_{31}+s_{21}.
$$
The general equation (4) follows again applying $\sigma:=\begin{pmatrix} 1&2&3 \\ i&j&k \end{pmatrix}$.

Finally we apply $\tau = (4,5)$: $A_{12,32} = \tau A_{13,31} = s_{21}-s_{12}+s_{32}-s_{23}$ and applying  $\sigma:=\begin{pmatrix} 1&2&3 \\ i&j&k \end{pmatrix}$ yields
$$
A_{ij,kj} = s_{ji}-s_{ij}+s_{kj}-s_{jk}.
$$
\end{proof}

\begin{proof} [Proof of Theorem \ref{mainresult}]
The matrix follows immediately from Proposition \ref{invariantmatrix}. A straightforward calculation shows that the 15 $(4 \times 4)$ - Pfaffians of the matrix $A$ are:
$$
P_1= 2Q_{12}-2Q_{13}, \ P_2= 2Q_{12}-2Q_{32}, \ P_3= 2Q_{21}-2Q_{12}, \ P_4= 2Q_{13}-2Q_{21}, 
$$
$$
P_5= 2Q_{23}-2Q_{13}, \ P_6= 2Q_{32}-2Q_{23}, \ P_7= 2Q_{21}-2Q_{32}, \ P_8= 2Q_{31}-2Q_{32}, 
$$
$$
P_9= 2Q_{31}-2Q_{21}, \ P_{10}= 2Q_{13}-2Q_{31}, \ P_{11}= 2Q_{31}-2Q_{12}, \ P_{12}= 2Q_{23}-2Q_{12}, 
$$
$$
P_{13}= 2Q_{21}-2Q_{23}, \ P_{14}= 2Q_{32}-2Q_{13}, \ P_{15}= 2Q_{12}-2Q_{23}.
$$
This proves the theorem.
\end{proof}

\section{$\mathfrak{S}_5$-invariant equations of  $Y \subset (\PP^1)^5$}

As already mentioned in the first section there are  five geometric objects permuted by the automorphism group $\mathfrak S_5$ of the del Pezzo surface $Y$ of degree five: namely, $5$ fibrations $\varphi_i : Y \ra \PP^1$,
induced, for $1 \leq i \leq 4$, by the projection with centre $p_i$, and, for $i=5$, by the pencil of conics through
the $4$ points (viewing the Del Pezzo surface as the moduli space of five ordered points on the projective line, these fibrations are just the maps to the moduli space of four ordered points on the projective line obtained forgetting the $j$-th of the five points). Each fibration  is a conic bundle, with exactly three singular fibres, correponding to the possible partitions of
type $(2,2)$ of the set $\{1, 2,3,4,5\} \setminus \{i\}$.

We have proven in \cite{debart} the following result

\begin{theorem}
 $Y$ embeds into $(\PP^1)^4$ via $\varphi_1 \times \dots \times  \varphi_4$
 and in $(\PP^1)^5$ via $\varphi_1 \times \dots \times \varphi_5$.
 \end{theorem}

Moreover, we showed \footnote{In the previous paper the coordinates which are here denoted $(t'_1,t'_2)$ were denoted $(t_1,t_2)$; whereas here
we reserve the notation $(t_1,t_2)$ for $(t'_1, - t'_2)$ in order to show the $\mathfrak{S}_5$-symmetry,}
 \begin{theorem}\label{eqdp}
 Let $\Sigma \subset (\PP^1)^4=:Q$, with coordinates 
 $$(v_1:v_2),(w_1:w_2),(z_1:z_2), (t'_1:t'_2),$$
  be the image of the Del Pezzo surface $Y$ via  $\varphi_3 \times \varphi_1 \times \varphi_2 \times   \varphi_4$.
 Then the equations of $\Sigma$ are given by the four $3 \times 3$-minors of the following Hilbert-Burch matrix:
 \begin{equation}
 A:=
 \begin{pmatrix}
 t'_2 & -t'_1 &t'_1+t'_2\\
  v_1 & v_2 & 0\\
 w_2 & 0 &w_1 \\
 0&-z_1&z_2 
 \end{pmatrix} .
\end{equation}
In particular, we have a Hilbert-Burch resolution:
\begin{equation}
0 \ra (\hol_Q(-\sum_{i=1}^4 H_i))^{\oplus 3} \ra \bigoplus_{j=1}^4( \hol_Q(-\sum_{i=1}^4 H_i + H_j)) \ra \hol_Q \ra \hol_{\Sigma} \ra 0,
\end{equation}
where $H_i$ is the pullback to $Q$ of a point in $\PP^1$ under the i-th projection.
 \end{theorem}

Observe first that each pencil $\varphi_i \colon Y \ra \PP^1$ can be rewritten as
 $$
 \varphi_i \colon Y \ra \Lambda_i \subset \PP^2,
 $$
 where $\Lambda_1, \Lambda_2, \Lambda_3, \Lambda_4, \Lambda_5$ are the lines in $\PP^2$ defined by
 equations which are consequences of  the following equalities: 
  \begin{eqnarray*}
 y_1:= x_3-x_2, \\ 
  y_2:=x_1-x_3, \\
  y_3:= x_2-x_1, \\  
y_1+y_2+y_3 = 0, \\
  x_1y_1 + x_2y_2 +x_3y_3 = 0.
 \end{eqnarray*}
 
 We set
 \begin{eqnarray*}
 \Lambda_1: =& \{ w_1-w_2+w_3 = 0\}, \\ 
 \Lambda_2: =& \{  z_1-z_2+z_3 = 0 \}, \\
 \Lambda_3: = & \{ v_1-v_2+v_3 = 0\}, \\ 
 \Lambda_4: = & \{   t_1-t_2+t_3  = 0\}, \\
 \Lambda_5 := &  \{   s_1-s_2+s_3  = 0\},
 \end{eqnarray*}
 
 and then the map  
 $$
 \varphi_3 \times \varphi_1 \times \varphi_2 \times \varphi_4 \times \varphi_5 \colon Y \ra \Lambda_3 \times \Lambda_1 \times \Lambda_2 \times \Lambda_4 \times \Lambda_5
 $$ is  expressed by: 
 \begin{eqnarray*}
(v_1,v_2,v_3):=(x_1,x_2,y_3),  & (w_1,w_2,w_3):=(x_2,x_3,y_1),\\  
(z_1,z_2,z_3):=(x_3,x_1,y_2),  & (t_1,t_2,t_3):=(y_1,-y_2,y_3), 
 \end{eqnarray*}
 $$
 (s_1,s_2,s_3):= (x_1y_1:-x_2y_2:x_3y_3). 
$$
 The equation of the image  of the Del Pezzo surface $Y$ in $(\PP^1)^3 =  \Lambda_3 \times \Lambda_1 \times \Lambda_2$ under the map map  $\varphi_3 \times \varphi_1 \times \varphi_2$ is then
 $$v_1w_1z_1 - v_2w_2z_2 = 0.
 $$

We give now the action of $\mathfrak{S}_5$ on the pencils $\varphi_i$ in order to determine the 10 ($\mathfrak{S}_5$-invariant) equations of the image of $Y$ under the map $\varphi_3 \times \varphi_1 \times \varphi_2 \times \varphi_4 \times \varphi_5$, which correspond to the 10 possible projections of $(\PP^1)^5$ to $(\PP^1)^3$.

In fact, $\mathfrak{S}_5$ acts by permuting the indices of $\varphi_i$, but a permutation $\tau \in \mathfrak{S}_5$ maps $\varphi_i$ to $\lambda_{\tau}  \circ \varphi_{\tau(i)}$, where $\lambda_{\tau}$ is a projectivity of $\PP^1$.

A straightforward computation (using the formulae for the action of $\mathfrak{S}_5$ on $\PP^2$ by birational maps given in section \ref{section1}) gives the following table.

\begin{tabular}{|c||c|c|c|c|}
\hline
\hline
$\tau$  & $(1,2)$ & $(2,3)$ & $(3,4)$ & $(4,5)$ \\
\hline
\hline 
$\tau(v_1:v_2:v_3)$ & $(v_2:v_1:-v_3)$  & $(z_2:z_1:-z_3)$& $(t_2:t_1:-t_3)$ & $(v_2:v_1:-v_3)$ \\
\hline
$\tau(w_1:w_2:w_3)$ & $(z_2:z_1:-z_3)$  & $(w_2:w_1:-w_3)$& $(w_3:w_2:w_1)$ &  $(w_2:w_1:-w_3)$ \\
\hline
$\tau(z_1:z_2:z_3)$ & $(w_2:w_1:-w_3)$  & $(v_2:v_1:-v_3)$& $(-z_1:z_3:z_2)$ &  $(z_2:z_1:-z_3)$ \\
\hline
$\tau(t_1:t_2:t_3)$ & $(t_2:t_1:-t_3)$  & $(-t_1:t_3:t_2)$& $(v_2:v_1:-v_3)$ &  $(s_1:s_2:s_3)$ \\
\hline
$\tau(s_1:s_2:s_3)$ & $(s_2:s_1:-s_3)$  & $(-s_1:s_3:s_2)$& $(s_2:s_1:-s_3)$ &  $(t_1:t_2:t_3)$ \\
\hline 
\hline
\end{tabular}

\begin{theorem}\label{product}
 Let $\Sigma \subset (\PP^1)^5=\Lambda_3 \times \Lambda_1 \times \Lambda_2 \times \Lambda_4 \times \Lambda_5 \subset (\PP^2)^5$, with coordinates 
 $$(v_1:v_2:v_3),(w_1:w_2:w_3),(z_1:z_2:z_3), (t_1:t_2:t_3), (s_1:s_2:s_3)$$
  be the image of the Del Pezzo surface $Y$ under the map  $\varphi_3 \times \varphi_1 \times\varphi_2 \times\varphi_4  \times  \varphi_5$.
 Then the equations of $\Sigma$ are the following:
 \begin{eqnarray*}
1) & v_1w_1z_1 - v_2w_2z_2 = 0, \\  
2) & v_3w_1t_1 - v_2w_3t_3 = 0, \\  
3) & v_1z_3t_3+v_3z_2t_2 = 0, \\
4) &w_3z_1t_2+w_2z_3t_1= 0, \\  
5) &t_1v_1s_2-t_2v_2s_1=0, \\
6) &t_1z_2s_3-t_3z_1s_1=0, \\
7) & t_3w_2s_2-t_2w_1s_3=0, \\
8) & v_3w_2s_1-v_1w_3s_3=0, \\
9) & v_2z_3s_3+v_3z_1s_2=0, \\
10) & w_3z_2s_2+w_1z_3s_1=0.
 \end{eqnarray*}

\end{theorem}

\begin{proof}
The equations $2) -10)$ are obtained from the first one using the above table described in the following diagram:

\begin{equation*}
\xymatrix{
1) \ar[r]^{(3,4)} &  4) \ar[r]^{(4,5)} &10) \\
4) \ar[r]^{(2,3)} &2) \ar[r]^{(1,2)}& 3) \ar[r]^{(4,5)}& 9) \ar[r]^{(3,4)} &6) \ar[r]^{(2,3)}& 5),\\
10) \ar[r]^{(2,3)} &  8) \ar[r]^{(3,4)}& 7) , &&& \\
}
\end{equation*}
\end{proof}

 \begin{remark}\label{p15}
 We have seen that the equations of 
 $$
 Y':=(\varphi_1 \times \ldots \times \varphi_5)(Y) \subset (\PP^1)^5
 $$
 are the ten equations obtained by the ten coordinate projections $P:=(\PP^1)^5 \ra (\PP^1)^3$ and therefore we have an exact sequence 
 \begin{multline}
 0 \ra (\hol_P(-\sum_{i=1}^5 H_i))^{\oplus 6} \ra (\bigoplus_{j=1}^5\hol_P(-\sum_{i=1}^5 H_i + H_j)) ^{\oplus 3} \ra \\
 \ra \bigoplus_{h<k}( \hol_P(-\sum_{i=1}^5 H_i + H_k +H_h)) \ra \hol_P \ra \hol_{Y'} \ra 0,
\end{multline}
where the first syzygies are the pull-backs of the syzygies obtained for each projection $(\PP^1)^5 \ra (\PP^1)^4$. 

Observe that the shape of this resolution is the same as the Eagon-Northcott complex for a Del Pezzo surface $S_4$ of degree 4, cf \cite{debart}.

But if this resolution were associated to a $5 \times 3$ matrix of the same type, then we would get a Del Pezzo surface of degree 4
and not of degree 5. 

Hence (since $K^2$ is invariant for smooth deformations) we have a Hilbert scheme which is reducible (since the open set  corresponding to
smooth surfaces is disconnected).
\end{remark}

\end{document}